\documentclass[a4paper]{amsart}

\usepackage[T1]{fontenc}
\usepackage{lmodern}
\usepackage{microtype}

\usepackage{mathtools}
\usepackage{amssymb}
\usepackage{amsthm}

\usepackage{graphicx}
\usepackage{tikz}

\usepackage{enumitem}

\usepackage[hidelinks]{hyperref}
\usepackage[nameinlink]{cleveref}

\theoremstyle{plain}
\newtheorem{theorem}{Theorem}[section]
\newtheorem{lemma}[theorem]{Lemma}

\newtheorem{corollary}[theorem]{Corollary}

\newtheorem*{unnumberedproblem}{Problem}

\theoremstyle{definition}
\newtheorem{definition}[theorem]{Definition}

\theoremstyle{remark}
\newtheorem{remark}[theorem]{Remark}

\newcommand{\qua}{\hskip 0.4em\ignorespaces}

\def\arxiv#1{%
  \relax
  \ifhmode\unskip\qua\fi
  \href{https://arxiv.org/abs/#1}%
       {\texttt{arXiv:\penalty-100\unskip#1}}%
}

\def\MR#1{%
  \relax
  \ifhmode\unskip\qua\fi
  \href{https://mathscinet.ams.org/mathscinet-getitem?mr=#1}%
       {\texttt{MR#1}}%
}

\def\ZB#1{%
  \relax
  \ifhmode\unskip\qua\fi
  \href{https://zbmath.org/?q=an:#1}%
       {\texttt{Zbl:#1}}%
}

\def\xox#1{\csname xx#1\endcsname}

\let\originalthebibliography\thebibliography
\let\endoriginalthebibliography\endthebibliography

\renewenvironment{thebibliography}[1]
  {%
    \originalthebibliography{MM20.}%
    \small
    \setlength{\itemsep}{0.5ex}%
    \setlength{\parskip}{0pt}%
  }
  {%
    \endoriginalthebibliography
  }

\ifdefined\pdfsuppresswarningpagegroup
\fi

\title{The support genus does not increase under contact connected sum}

\author{Miguel Orbegozo Rodriguez}
\email{miguel.orbegozorodriguez@gmail.com}
\urladdr{https://sites.google.com/view/miguel-orbegozo-rodriguez/home}

\author{Eric Stenhede}
\email{eric.stenhede@univie.ac.at}

\begin{document}

\begin{abstract}
We show that the support genus of the contact connected sum of two contact $3$-manifolds is at most the maximum of the support genera of the summands. In particular, iterated connected sums of contact manifolds of support genus one still have support genus at most one, and therefore cannot provide candidates for contact manifolds of higher support genus.
\end{abstract}

\maketitle

\section{Introduction}

Giroux's correspondence allows us to study cooriented contact structures on
closed $3$-manifolds through open books~\cite{Giroux02,BHH24,LicVer2_24}. The \emph{support genus} of a contact structure $\xi$, denoted by $\operatorname{sg}(\xi)$, is the minimum genus of a page of an open book supporting $\xi$. This invariant was introduced by Etnyre and Ozbagci in
\cite{EtnyreOzbagci}.\footnote{They note that it had already appeared
implicitly in \cite{PlanarEtnyre}.} We note the following.

\begin{theorem}\label{thm:main}
Let $(M_i,\xi_i)$, $i=1,2$, be closed, connected, cooriented contact
$3$-manifolds. Then
\[
    \operatorname{sg}(\xi_1\#\xi_2)
    \leq
    \max\{\operatorname{sg}(\xi_1),\operatorname{sg}(\xi_2)\},
\]
where $\xi_1\#\xi_2$ denotes the contact structure on $M_1\#M_2$ obtained by taking the contact connected sum.
\end{theorem}

We remark that if either summand is overtwisted, the result is already known: the contact
connected sum is also overtwisted and hence planar~\cite{PlanarEtnyre}. This also shows that
the inequality in Theorem~\ref{thm:main} can be strict.

Much of the literature on the support genus concerns planar open books. Etnyre gave obstructions to planarity~\cite{PlanarEtnyre} by deriving restrictions on the symplectic fillings of planar contact manifolds; further obstructions were
developed in
\cite{PlanarO-S-Sz,planarWendl,planarAndy,PlanarGh-Go-Pl}. These results distinguish support genus zero from positive support genus,
but they do not distinguish support genus one from higher support genus.
The existence of contact structures of support genus greater than one is
a long standing question. Already in the paper in which the support genus was
introduced, Etnyre and Ozbagci observed that, despite the existence of many
potential examples, no contact structure of support genus greater than one
was known~\cite[Section~4]{EtnyreOzbagci}. More recently, the question was
recorded as Problem~3.45 in the K3 problem list.

\begin{unnumberedproblem}[{\cite[Problem 3.45]{K3}}]
Are there contact $3$-manifolds with support genus greater than one?
Are there contact $3$-manifolds with arbitrarily large support genus?
\end{unnumberedproblem}

The problem list suggests two possible approaches. The first is to take
iterated connected sums of contact structures of support genus one. The second is to consider contact structures supported
by open books $(S,\phi)$ for which $S$ has large genus and connected
boundary and $\phi$ is a positive power of a Dehn twist parallel to
$\partial S$; see also~\cite{MassotSupportGenus}. Theorem~\ref{thm:main} rules out the first strategy, as recorded by the following corollary.

\begin{corollary} \label{cor:genus1}
    For $i = 1,...,n$, let $(M_i,\xi_i)$ be closed, connected, cooriented contact manifolds such that $sg(\xi_i) = 1$. Then $\#_{i=1}^n (M_i,\xi_i)$ has support genus at most 1.  \qed
\end{corollary}

The proof of Theorem~\ref{thm:main} is short. We construct a polygonal Murasugi sum of two compact
surfaces whose genus is the maximum of their genera, and then apply
Torisu's theorem relating Murasugi sum and contact connected
sum~\cite{Contact_murasugi_sum}.

\subsection*{Acknowledgements} The authors would like to thank Peter Feller for helpful conversations, and comments on a first draft of this paper. Miguel Orbegozo Rodriguez is supported by the Swiss NSF grant 200021-212085. Eric Stenhede is supported by the Austrian Science Fund (FWF) project PAT7436924.

\section{Murasugi sums and the proof of Theorem 1.1}

We first recall the Murasugi sum of compact surfaces and of abstract
open books.

\begin{definition}\label{def:murasugi-sum}
Let $\Sigma_1$ and $\Sigma_2$ be compact, connected, oriented surfaces
with nonempty boundary. For $i=1,2$, let $P_i\subset\Sigma_i$ be an
embedded $2n$-gon whose edges alternate between arcs contained in
$\partial\Sigma_i$ and arcs with interior contained in
$\operatorname{int}(\Sigma_i)$. Identify $P_1$ with $P_2$ by an
orientation-preserving diffeomorphism which sends boundary edges of one
polygon to interior edges of the other. After smoothing corners, the
surface
\[
    \Sigma=\Sigma_1\cup_{P_1=P_2}\Sigma_2
\]
is called the \emph{Murasugi sum} of $\Sigma_1$ and $\Sigma_2$ and is denoted by $\Sigma_1*_P\Sigma_2$. Here $P$ is the common image of $P_1$ and $P_2$ in $\Sigma$.

Let $(\Sigma_i,\varphi_i)$, $i=1,2$, be abstract open books. Let
$\overline{\varphi}_i\colon\Sigma\to\Sigma$ be the extension of $\varphi_i$ by the
identity outside $\Sigma_i$. The abstract open book
\[
    \bigl(\Sigma,
    \overline{\varphi}_1\circ\overline{\varphi}_2\bigr)
\]
is called the \emph{Murasugi sum} of
$(\Sigma_1,\varphi_1)$ and $(\Sigma_2,\varphi_2)$. In the interest of clarity we suppress $P$ from the notation and denote this Murasugi sum by $(\Sigma_1,\varphi_1)*(\Sigma_2,\varphi_2)$.
\end{definition}

The following is the key construction that we need to prove Theorem \ref{thm:main}.

\begin{lemma}\label{lem:genus-preserving-murasugi-sum}
Let $\Sigma_i$, $i=1,2$, be compact, connected, oriented surfaces
with nonempty boundary. Then there exists a Murasugi sum $\Sigma$ of
$\Sigma_1$ and $\Sigma_2$ such that
\[
    g(\Sigma)=\max\{g(\Sigma_1),g(\Sigma_2)\}.
\]
\end{lemma}

\begin{proof}
Write $g_i=g(\Sigma_i)$ and, after possibly exchanging the two
surfaces, assume that $g_1\geq g_2$.

Suppose first that at least one of the two surfaces is planar. A
boundary connected sum is a particular type of Murasugi sum, and the
genus of a boundary connected sum is the sum of the genera of its
summands. Since one of the two genera is zero, the resulting surface
has genus $\max\{g_1,g_2\}$.

We now assume that $g_1,g_2>0$. We first describe the construction when
the two surfaces have the same genus and connected boundary.

Consider two genus one surfaces with connected boundary. The
desired Murasugi sum is shown in
\Cref{fig:genus-one-murasugi-sum}. The polygons
$P_i\subset\Sigma_i$, $i=1,2$, along which we perform the Murasugi sum
are shown in blue. Their sides alternate between arcs contained in
$\partial\Sigma_i$ and arcs whose interiors are contained in
$\operatorname{int}(\Sigma_i)$. The identification of $P_1$ with $P_2$
interchanges these two types of sides and therefore defines a Murasugi
sum. We denote by $P\subset\Sigma$ the common image of $P_1$ and $P_2$
in the resulting surface. As is visible in the figure, $\Sigma$ has genus
one. This can also be seen by computing $\chi(\Sigma)=-3$ and observing
that $\Sigma$ has three boundary components.

\begin{figure}
    \centering
    \includegraphics[width=0.5\linewidth]{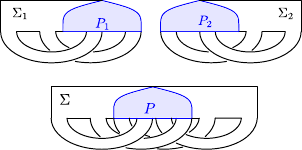}
    \caption{A Murasugi sum of two genus-one surfaces with connected
    boundary whose resulting surface also has genus one. The Murasugi
    polygons $P_1$ and $P_2$, and their common image $P$, are shown in
    blue.}
    \label{fig:genus-one-murasugi-sum}
\end{figure}

We next consider two surfaces with connected boundary and equal genus
$g>1$. The construction is shown schematically in
\Cref{fig:equal-genus-murasugi-sum}. The surface $\Sigma_1$ is represented in the top row, $\Sigma_2$ in the middle one, and $\Sigma$ in the bottom one. Each of the surfaces is
obtained by joining a leftmost piece to a rightmost piece and inserting
$g-2$ copies of the central piece between them. Thus, when no central
piece is inserted, the corresponding surface has genus two, and each
additional central piece increases its genus by one. It is easy to see that $\Sigma_1$ and $\Sigma_2$ have genus $g$.

The blue regions in the individual pieces join to form single embedded
polygons
\[
    P_i\subset\Sigma_i,\qquad i=1,2.
\]
Their sides alternate between arcs contained in $\partial\Sigma_i$ and
arcs whose interiors are contained in $\operatorname{int}(\Sigma_i)$.
As in the genus one case, the identification of $P_1$ with $P_2$
interchanges these two types of sides and hence defines a Murasugi sum.
The blue regions in the third row similarly join to form the common
image $P\subset\Sigma$.

The resulting surface $\Sigma$ is the boundary connected sum of $g$
copies of the genus one surface $\Sigma$ appearing in
Figure \ref{fig:genus-one-murasugi-sum}. Since the genus of a boundary
connected sum is the sum of the genera of its summands, it follows that
$g(\Sigma)=g$.

\begin{figure}
    \centering
    \includegraphics[width=0.9\linewidth]{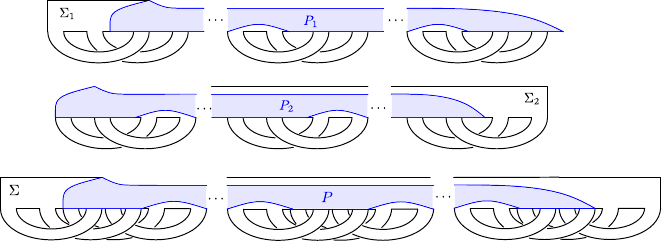}
    \caption{A Murasugi sum of two genus-$g$ surfaces with connected
    boundary whose resulting surface also has genus $g$. In each row,
    $g-2$ copies of the central piece are inserted between the leftmost
    and rightmost pieces. The blue regions join to form the Murasugi
    polygons $P_1$ and $P_2$ and their common image $P$.}
    \label{fig:equal-genus-murasugi-sum}
\end{figure}

We have thus proved the claim when the two surfaces have the same genus
and connected boundary.

We now reduce the general case to this construction. Let
$g_1\geq g_2>0$, and let $b_i$ denote the number of boundary components
of $\Sigma_i$. Choose surfaces of genus $g_2$ with connected boundary
$\widetilde{\Sigma}_1$ and $\widetilde{\Sigma}_2$, equipped with the
Murasugi polygons $P_1\subset\widetilde{\Sigma}_1$ and
$P_2\subset\widetilde{\Sigma}_2$ constructed as in the previous case. Choose also a
surface $R$ of genus $g_1-g_2$ with $b_1$ boundary components and a
planar surface $S$ with $b_2$ boundary components. Then
\[
    \Sigma_1\cong R\natural\widetilde{\Sigma}_1
    \qquad\text{and}\qquad
    \Sigma_2\cong S\natural\widetilde{\Sigma}_2,
\]
where $\natural$ denotes boundary connected sum.

We choose the boundary connected sums so that their attaching
regions are disjoint from $P_1$ and $P_2$. Consequently, the
polygons $P_1$ and $P_2$ remain Murasugi polygons in $\Sigma_1$ and
$\Sigma_2$. Performing the Murasugi sum along these polygons produces a
surface diffeomorphic to
\[
    R\natural
    \bigl(\widetilde{\Sigma}_1*_{P}
          \widetilde{\Sigma}_2\bigr)
    \natural S.
\]
The middle surface has genus $g_2$ by the considerations in the equal genus case
above. Since genus is additive under boundary connected sum and $S$ is
planar, the genus of the resulting surface is
\[
    (g_1-g_2)+g_2+0=g_1
    =\max\{g_1,g_2\}.
\]
This proves the result.

\end{proof}

With this in hand we can now prove Theorem \ref{thm:main}.

\begin{proof}[Proof of Theorem \ref{thm:main}]
For $i=1,2$, choose an open book $(\Sigma_i,\varphi_i)$ supporting
$(M_i,\xi_i)$ such that $g(\Sigma_i)=\operatorname{sg}(\xi_i)$. By Lemma~\ref{lem:genus-preserving-murasugi-sum}, there is a Murasugi sum $\Sigma$ between $\Sigma_1$ and $\Sigma_2$ such that $g(\Sigma)=\max\{g(\Sigma_1),g(\Sigma_2)\}$.

Let $(\Sigma,\Phi)=(\Sigma_1,\varphi_1)*(\Sigma_2,\varphi_2)$ be the corresponding Murasugi sum of abstract open books. By a theorem of Torisu~\cite[Theorem 5.4]{Contact_murasugi_sum},\footnote{See also \cite{MR3460043} for a reformulation of the result in terms of abstract open books.} $(\Sigma,\Phi)$ supports the contact connected sum
\[(M_1,\xi_1)\#(M_2,\xi_2).
\]
Consequently,
\[
\begin{aligned}
    \operatorname{sg}(\xi_1\#\xi_2)
    &\leq g(\Sigma)\\
    &=\max\{g(\Sigma_1),g(\Sigma_2)\}\\
    &=\max\{\operatorname{sg}(\xi_1),
             \operatorname{sg}(\xi_2)\}.
\end{aligned}
\]
\end{proof}

We conclude by observing that the bound in Theorem~\ref{thm:main}
cannot be improved by taking open books for
$\xi_1$ and $\xi_2$ and choosing a different Murasugi sum of their
pages. This is the content of the following lemma.\footnote{We did not find a reference for the proof of Lemma \ref{lem:murasugi-genus-lower-bound} so we include our own proof.}

\begin{lemma}\label{lem:murasugi-genus-lower-bound}
Let $\Sigma_1$ and $\Sigma_2$ be compact, connected, oriented surfaces
with nonempty boundary, and let $g_i=g(\Sigma_i)$ for $i=1,2$. Then any
Murasugi sum $\Sigma$ of $\Sigma_1$ and $\Sigma_2$ satisfies
\[
    g(\Sigma)\geq \max\{g_1,g_2\}.
\]
\end{lemma}

\begin{proof}
After possibly exchanging the two summands, assume that $g_1\geq g_2$.
By the definition of a Murasugi sum, $\Sigma_1$ embeds as an
orientation-preserving subsurface of $\Sigma$.

Let $\widehat{\Sigma}$ be the closed surface obtained by capping every
boundary component of $\Sigma$ with a disc. Notice that
$g(\widehat{\Sigma})=g(\Sigma)$.

Choose oriented simple closed curves
\[
    a_1,b_1,\ldots,a_{g_1},b_{g_1}
    \subset \operatorname{int}(\Sigma_1)
\]
whose algebraic intersection
numbers satisfy
\[
    \hat\imath([a_i],[b_j])=\delta_{ij},
    \qquad
    \hat\imath([a_i],[a_j])
    =
    \hat\imath([b_i],[b_j])
    =0.
\]
This is possible because $\Sigma_1$ has genus $g_1$. Regarded as curves in $\widehat{\Sigma}$, they have the same algebraic
intersection numbers. Therefore, their homology classes are linearly
independent in $H_1(\widehat{\Sigma};\mathbb Z)$.

It follows that
\[
    \operatorname{rank}H_1(\widehat{\Sigma};\mathbb Z)\geq 2g_1.
\]
Since $\widehat{\Sigma}$ is a closed surface of genus $g(\Sigma)$, we
have
\[
    \operatorname{rank}H_1(\widehat{\Sigma};\mathbb Z)=2g(\Sigma).
\]
Therefore $g(\Sigma)\geq g_1=\max\{g_1,g_2\}$.
\end{proof}

\begin{remark}
Lemma~\ref{lem:murasugi-genus-lower-bound} shows that the construction
used in the proof of Theorem~\ref{thm:main} has the smallest possible
genus among all Murasugi sums of supporting open books for the two
summands. It does not rule out the existence of a supporting open book
of smaller genus.
\end{remark}

\bibliographystyle{myamsalpha} 
\bibliography{main}

\end{document}